\documentclass[10pt, reqno]{amsart}
\usepackage{amsmath,amssymb,amsfonts, amscd,hyperref}
\usepackage[all,cmtip]{xy}
\usepackage[usenames,dvipsnames]{color}
\usepackage[small,nohug,heads=vee]{diagrams}
\diagramstyle[labelstyle=\scriptstyle]
\usepackage{epsfig}
\usepackage{graphicx}
\usepackage{verbatim}
\usepackage{amssymb}
\usepackage{amsfonts}
\usepackage{amsbsy}
\usepackage{latexsym}
\usepackage{amsthm}
\usepackage{amstext}
\usepackage{amsopn}
\usepackage{indentfirst}
\usepackage{mathrsfs}
\usepackage{amsmath}
\usepackage{ulem}
\usepackage[T1]{fontenc}
\usepackage{eucal,mathrsfs,dsfont}
\usepackage{color}

\definecolor{mno}{rgb}{0.5,0.1,0.5}

\newcommand{\Hmm}[1]{\leavevmode{\marginpar{\tiny%
$\hbox to 0mm{\hspace*{-0.5mm}$\leftarrow$\hss}%
\vcenter{\vrule depth 0.1mm height 0.1mm width \the\marginparwidth}%
\hbox to 0mm{\hss$\rightarrow$\hspace*{-0.5mm}}$\\\relax\raggedright #1}}}

\def\De{\Delta}

\newtheorem{theorem}{Theorem}
\newtheorem{thm}{Theorem}[section]
\newtheorem{cor}{Corollary}
\newtheorem{lem}[thm]{Lemma}
\newtheorem{lemma}[thm]{Lemma}
\newtheorem{pro}[thm]{Proposition}

\theoremstyle{definition}
\newtheorem*{defi}{Definition}
\newtheorem{eg}[thm]{Example}
\newtheorem{rem}[thm]{Remark}

\newcommand{\Z}{{\mathbb Z}}
\newcommand{\R}{{\mathbb R}}

\newcommand{\N}{{\mathbb N}}

\newcommand{\al}{{\alpha}}
\newcommand{\be}{{\beta}}
\newcommand{\eps}{{\varepsilon}}
\newcommand{\gm}{{\gamma}}
\newcommand{\Gm}{{\Gamma}}
\newcommand{\si}{{\sigma}}
\newcommand{\lm}{{\lambda}}
\newcommand{\Deg}{{\mathrm{Deg}}}
\newcommand{\supp}{{\mathrm{supp}\,}}

\newcommand{\cb}{{\partial_C X}}
\newcommand{\cd}{\overline{\rm codim}_M}

\newcommand{\as}[1]{\left\langle #1\right\rangle}

\newcommand{\aV}[1]{\left\Vert #1\right\Vert}
\newcommand{\ov}[1]{\overline{ #1}}
\newcommand{\ow}[1]{\widetilde{ #1}}

\newcommand{\qn}{Q^{ \max}}

\newcommand{\cp}{{\rm Cap}}

\newcommand{\n}{n}

\begin{document}
\title[Self-adjoint extensions of graph Laplacians]
{A note on self-adjoint extensions of the Laplacian on weighted graphs }

\author[X. Huang]{Xueping Huang}
\address{Fakult\"at f\"ur Mathematik, Universit\"at Bielefeld,  33501 Bielefeld, Germany}
\email{xhuang1@math.uni-bielefeld.de}

\author[M. Keller]{Matthias Keller}
\address{Mathematisches Institut, Friedrich-Schiller-Universit\"at Jena,
07743 Jena, Germany}
\email{m.keller@uni-jena.de}

\author[J. Masamune]{Jun Masamune}
\address{Research Center for Pure and Applied Mathematics
Graduate School of Information Sciences
Tohoku University
6-3-09 Aramaki-Aza-Aoba, Aoba-ku, Sendai 980-8579, Japan}
\email{jmasamune@m.tohoku.ac.jp}

\author[R. Wojciechowski]{Rados{\l}aw K. Wojciechowski}
\address{Department of Mathematics and Computer Science,
York College of the City University of New York, 94 - 20 Guy R. Brewer Blvd., Jamaica, NY 11451, USA}
\email{rwojciechowski@gc.cuny.edu}

\date{\today}
\maketitle


\begin{abstract} \noindent We study the uniqueness of self-adjoint and Markovian extensions of the Laplacian on weighted graphs. We first show that, for locally finite graphs and a certain family of metrics, completeness of the graph implies uniqueness of these extensions. Moreover, in the case when the graph is not metrically complete and the Cauchy boundary has finite capacity, we characterize the uniqueness of the Markovian extensions.
\end{abstract}

 \section{Introduction}

Determining the uniqueness of self-adjoint extensions of a symmetric operator in a certain class is a fundamental topic of functional analysis going back to the work of Friedrichs and von Neumann \cite{Frie, von}.
  If an operator has a unique self-adjoint extension, then it is called \textit{essentially self-adjoint}.  A self-adjoint extension whose corresponding form is a Dirichlet form is called \textit{Markovian} and, when such an extension is unique, the operator is said to have a \textit{unique Markovian extension}.  It is clear that essential self-adjointness implies the uniqueness of Markovian extensions, but the converse is not necessarily true as can be seen by examples.

In the case of Riemannian manifolds, the (minimal) Laplacian,
whose domain is the space of smooth functions with compact support,
has Markovian extensions and generates the Brownian motion.
(The Laplacian should satisfy, in addition, the regularity property, but there
is always an equivalent operator which has this property \cite{Fuk}.)
A well-known result going back to the work of Gaffney \cite{Gaf.51, Ga} essentially states that,
on a geodesically complete manifold, the Laplacian has a
unique Markovian extension. (In \cite{Gaf.51, Ga} the so-called Gaffney Laplacian was proven to be essentially self-adjoint instead of the minimal one. The essential self-adjointness of the Gaffney Laplacian is equivalent to the uniqueness of Markovian extensions of the minimal Laplacian.  Indeed, Gaffney's result states that the condition $W_0^{1,2} = W^{1,2} $,
which is equivalent to the uniqueness of Markovian extensions of the minimal Laplacian \cite{GM},
implies the essential self-adjointness of the Gaffney Laplacian.  The converse implication was proven in \cite{Ma1}.
)
Later, it was shown that the Laplacian on a metrically complete Riemannian manifold is essentially self-adjoint \cite{Che, Str}. 
 On the other hand, if the manifold
is geodesically incomplete, the Laplacian is not essentially self-adjoint in general;
however, if the Cauchy boundary, which is the difference
between the completion
of the manifold and the manifold itself, is ``small\rq\rq{} in some sense,
then the Laplacian is essentially self-adjoint or has a unique Markov extension
 depending on how small the Cauchy boundary is
\cite{Chee, VERDIERE, GM, Kuw, LT, MASAMUNE} (see also the references within).
For strongly local regular Dirichlet forms, the
uniqueness of Silverstein extensions was proven by
Kawabata and Takeda \cite{KT} in the case when the underlying space is metrically complete  with respect to the Carnot-Caratheodori distance.
This result was extended to general regular Dirichlet forms by
Kuwae and Shiozawa \cite{KS} using the intrinsic distance defined by Frank, Lenz, and Wingert in \cite{FLW}.

Recently, there has been a tremendous amount of work devoted to the study of self-adjoint extensions of certain operators defined on graphs.  More specifically, these issues are studied for adjacency, (magnetic) Laplacian, and Schr{\"o}dinger-type operators on graphs in \cite{dVTHT, dVTHT2, Dod, Gol, Gol2, GS, HKLW, Jor, JP, KL, KL2, Kuw, Ma2, Mi, Mi2, TH, Web, Woj1, Woj2} among others.

Let us mention, in particular, the series of papers \cite{dVTHT, dVTHT2, TH} by Colin de Verdi\`{e}re, Torki-Hamza, and  Truc.  These papers give some relations between metric completeness and essential self-adjointness. However, \cite{HKLW} contains an example of a graph which is metrically complete in one of the distances studied in \cite{dVTHT} but for which the corresponding weighted Laplacian does not have a unique Markovian extension and is, therefore, not essentially self-adjoint.  One reason for this seems to be that the particular metric used in \cite{dVTHT} does not take into account the measure on the vertices of the graph.  In \cite{ Mi,Mi2}, Milatovic, following \cite{TH}, shows, with a different metric, that completeness  implies essential self-adjointness under the additional assumption of a uniform bound on the vertex degree.  \\

In this paper we investigate these questions for the weighted Laplacian on graphs. Recall that the weighted Laplacian has Markovian extensions and the associated form is one of the most important examples of a non-local Dirichlet form.
We use the notion of intrinsic distance
 introduced in \cite{FLW}
and show that, if a weighted degree function is bounded on the combinatorial neighborhood of each ball defined with respect to one such distance, then the Laplacian is essentially self-adjoint (Theorem~\ref{thm-main}).   As a direct consequence, in the locally finite case, if the graph is metrically complete in one intrinsic  path metric, then the Laplacian is essentially self-adjoint (Theorem~\ref{thm-locfinite}).  Compared to the previous results mentioned above we do not assume a uniform bound on the vertex degree and, for Theorem~\ref{thm-main}, we do not even need local finiteness. These results indicate that intrinsic metrics  give the correct notion of distance on graphs when seeking to prove statements analogous to the strongly local case.

In the metrically incomplete case, under some further assumptions, we show that if the Cauchy (or metric) boundary has finite capacity, then the Laplacian has a unique Markovian extension if and only if the Cauchy boundary is polar, that is, has zero capacity, in analogy with \cite{MASAMUNE} (Theorem~\ref{thm-locfinite2}). Moreover, we show that upper Minkowski codimension of the boundary greater than 2 implies zero capacity of the boundary (Theorem~\ref{thm-Minkowski}).  We also show by examples that the other implications do not hold. In particular, in the case when the boundary has infinite capacity, the Laplacian may be essentially self-adjoint or might fail to have a unique Markovian extension, see Examples~\ref{Ex;2.2} and~\ref{Ex;2.3}.
In general, if the Laplacian is essentially self-adjoint, then it has a unique Markovian extension, but the opposite implication is not necessarily true, see Example~\ref{Ex;2.1}. In Examples~\ref{Ex;2.4},~\ref{Ex;2.5}, and ~\ref{Ex;2.6} we discuss the case of  upper Minkowski codimension less than or equal to 2 where the boundary may be polar or non-polar.

 The paper is organized as follows.
 In Section~\ref{SU}, we introduce the set up, including background material on
 Dirichlet forms, Laplacians, intrinsic  distances, and Cauchy boundary;
 and state the main results.
  In Section~\ref{US}, we establish the triviality of square integrable eigenfunctions with
 negative eigenvalue when the weighted degree function is bounded on the combinatorial neighborhood of each ball and use this to prove Theorems~\ref{thm-main} and~\ref{thm-locfinite}.
 Section~\ref{PM} is devoted to the proofs of Theorems~\ref{thm-locfinite2} and \ref{thm-Minkowski} and Section~\ref{s:examples} is devoted to (counter-)examples.
 In Appendix~\ref{PML},
 we prove a Hopf-Rinow type property for path metrics on locally  finite graphs. This property is used in the proof of Theorems~\ref{thm-locfinite} and~\ref{thm-locfinite2}.
 We also present a series of (counter-)examples showing that the property may fail
 if the graph is not locally finite.


\section{The set up and main results} \label{SU}
\subsection{Weighted graphs}
We generally follow the setting of \cite{KL}. Let
$X$ be a countably infinite discrete set. A
function $\mu:X\to(0,\infty)$ can be viewed as a Radon measure on $X$ with full support so that $(X,\mu)$ becomes a measure space.

Let $w:X\times X\to[0,\infty)$ be symmetric, with zero diagonal, and satisfying
\begin{align*}
\sum_{y\in  X}w(x, y)<\infty \quad \textup{ for all } x\in X.
\end{align*}
The triple $(X, w, \mu)$ is called a \textit{weighted graph}.
We call $x,y\in X$ \textit{neighbors} if $w(x,y)>0$ and denote this symmetric relation
by $x\sim y$.  If each vertex has only finitely many neighbors, then the graph is called \textit{locally finite}.
For $n\geq 1$, we call a sequence of points $(x_0 ,
x_1, \ldots, x_n)$ a \textit{path} connecting $x$ and $y$ if
$x_0 =x, x_n =y, $ and $ x_i\sim x_{i+1}$ for $i = 0, 1, \ldots, n-1$.
A weighted graph $(X, w, \mu)$ is called \textit{connected}
if, for any two distinct points in $X$,
there exists a connecting path.
From now on, we only consider connected weighted graphs.

\subsection{Weighted degree and intrinsic metrics}\label{ss:adapted}
We call the function $\mathrm{Deg}:X\to[0,\infty)$ given by
$$\mathrm{Deg}(x):= \frac{1}{\mu(x)}\sum_{y\in
 X}w(x, y)$$
the \textit{weighted degree}. It is, in general, distinct from
the combinatorial degree of locally finite graphs which is given
by the number of neighbors of a vertex.

A \textit{pseudo metric} is a map $d:X\times X\to[0,\infty)$
that is symmetric, has zero diagonal and satisfies the triangle inequality. A pseudo metric $d=d_{\si}$ is called a \textit{path pseudo metric} if there is a symmetric map $\si:X\times X\to[0,\infty)$ such that $\si(x,y)>0$  if and only if $x\sim y$ and
\begin{align*}
 d_{\si}(x,y)= \inf\{l_{\si}((x_0,\ldots, x_{n})) \mid n\geq 1, (x_0,\ldots, x_n) \text{~~~~is a path connecting $x$ and $y$} \}
\end{align*}
where the \textit{length} $l_{\si}$ of a path $(x_{0},\ldots,x_{n})$ is given by
$$l_{\si}((x_{0},\ldots,x_{n}))=\sum_{i=0}^{n-1} \si(x_{i},x_{i+1}).$$
We say that a pseudo metric $d$ has \textit{jump size} $s> 0$ if, for all $x,y\in X$, $w(x,y)=0$ whenever $d(x,y)>s$.

Following Frank/Lenz/Wingert \cite{FLW} (see Lemma 4.7 and Theorem 7.3)  we make a definition which has already proven to be useful in several other problems on graphs, see Remark~\ref{rem:intrinsic} below.

\begin{defi}
We call a pseudo metric $d$ on $(X, w, \mu)$ \textit{intrinsic} if, for all $x \in X$,
\begin{equation*}\frac{1}{\mu(x)}\sum_{y\in X}w(x,y)d(x,y)^2\leq
1.\end{equation*}
An intrinsic path pseudo metric $d_\si$ is called  \textit{strongly intrinsic} if, for all $x \in X$,
\[ \frac{1}{\mu(x)}\sum_{y\in X}w(x,y)\si(x,y)^2\leq 1. \]
\end{defi}

The first example below shows that
there always exist strongly intrinsic  path pseudo metrics   with jump size $1$ on
a connected weighted graph. 

\begin{eg} \label{adapted}
(1) For $x,y\in X$ with $x\sim y$, let
$\si_{0}(x,y)=\min \{{\mathrm{Deg}}^{-\frac{1}{2}} (x), {\mathrm{Deg}}^{-\frac{1}{2}} (y), 1 \}$.
Clearly,  $d_{\sigma_{0}}$ is strongly intrinsic with jump size 1.

(2) For locally finite graphs, let $\sigma_{1}(x,y)={w(x,y)^{-\frac{1}{2}}} \min\{ \frac{\mu(x)}{\deg(x)},\frac{\mu(y)}{\deg(y)}\}^{\frac{1}{2}}$, $x,y\in X$ with $x\sim y$ where $\deg$ is the combinatorial degree, i.e., the number of neighbors.
Clearly, $d_{\si_{1}}$ is a strongly intrinsic path metric. Moreover, if $\deg \le K$ for some $K\geq1$, then $d_{\si_{1}}$ is equivalent to the metrics used in \cite{dVTHT,dVTHT2,Mi,Mi2} (in the case of no magnetic field and no potential). This seems to explain why the combinatorial vertex degree has to be bounded for these results.

(3)  Suppose that $\sigma_{N}\equiv 1$ on neighbors. Then, $d_{N}=d_{\sigma_{N}}$ gives the natural graph metric, that is, the distance between $x$ and $y$ is equal to one less than the number of points in the shortest path connecting them. Obviously, $d_{N}$ is strongly intrinsic  if and only if $\Deg\leq 1$. (Clearly, if $\Deg$ is bounded by $K>0$, then $d_{N}/\sqrt{K}$ is also a strongly intrinsic  metric.)

\end{eg}

\begin{rem}\label{rem:intrinsic} Various authors came up with such types of metrics independently of \cite{FLW}. In the context of  stochastic completeness for jump processes, see  the work of  Masamune/Uemura \cite{MU}, Grigor'yan/Huang/Masamune \cite{GHM} and also \cite{Hu, H}. Independently,   Folz \cite{FOLZ} came up with similar ideas in the context of heat kernel estimates on locally finite graphs,  see also \cite{FOLZ2,FOLZ3}. For further uses of intrinsic  metrics, see \cite{HKW}.
\end{rem}

For $x_{0}\in X$ and $r\geq 0$, we define the distance ball with respect to any pseudo metric $d$ by     $B_r(x_0):=\{x\in X\mid d(x,x_{0})\leq r\}$.

\subsection{Forms and operators}
In this article, we only consider real valued functions.
Denote by $C(X)$ the set of all functions $X\to\R$ and by $C_{c}(X)$ the subset of functions which are finitely supported. The Hilbert space $L^2(X,\mu)$ is defined in the usual way with scalar product 
\begin{align*}
    \as{u,v}:=\sum_{X}uv\mu:=\sum_{x\in X}u(x)v(x)\mu(x)
\end{align*}
and norm     $\aV{u}:=  \as{u,u}^{\frac{1}{2}}=\left(\sum_{X}u^{2}\mu \right)^{\frac{1}{2}}$.

We next introduce a discrete version of  the energy measure which can be thought of as a generalized gradient. For $f\in C(X)$ and $x\in X$ define the square of the generalized gradient by
\begin{align*}
|\nabla f|^{2}(x)&:=\sum_{y\in X}w(x,y)(f(x)-f(y))^{2},
\end{align*}
which might take the value $\infty$.
For $x\in X$, let  $D_{\mathrm{loc}}(x):=\{(f,g)\in C(X)\times C(X)\mid \sum_{y\in X}w(x,y)|f(x)-f(y)||g(x)-g(y)|<\infty\}$ and for $(f,g)\in D_{\mathrm{loc}}(x)$, we define
\begin{align*}
(\nabla f\cdot\nabla g )(x)&:=\sum_{y\in X}{w(x,y)}(f(x)-f(y)){(g(x)-g(y))}.
\end{align*}

The \textit{generalized form} $\ow Q$ is a map $C(X)\to[0,\infty]$ given by
\begin{align*}
\ow Q(f):=\frac{1}{2}\sum_{X}|\nabla f|^{2}=\frac{1}{2}\sum_{x,y\in X}w(x,y)(f(x)-f(y))^{2}
\end{align*}\
and the \textit{generalized form domain} is given by $ \ow D:=\{f\in C(X)\mid \ow Q(f)<\infty\}.$ Clearly, $C_c(X)\subseteq \ow D$.  
By polarization, this gives a sesqui-linear form $\ow Q:\ow D\times\ow D\to\R$ as follows
\begin{align*}
\ow Q(f,g)=\frac{1}{2}\sum_{X}(\nabla f\cdot\nabla g)=\frac{1}{2}\sum_{x,y\in X}w(x,y)(f(x)-f(y))(g(x)-g(y)).
\end{align*}

In this context, there are two distinguished restrictions of the generalized form.  Let $Q$ be the restriction of $\ow Q$ to
\[ D(Q):=\overline{C_{c}(X)}^{\aV{\cdot}_{\widetilde{Q}}} \mbox{ where } \aV{\cdot}_{\widetilde{Q}}:=(\ow Q(\cdot)+\aV{\cdot}^2)^{\frac{1}{2}}. \]
The form $(Q,D(Q))$ is then a regular Dirichlet form, see \cite{KL}.
Furthermore, let $\qn$ be the restriction of $\ow Q$ to
\[ D(\qn):=\{f \in L^{2}(X,\mu) \mid \ow Q(f)<\infty\}. \]
The form $(\qn, D(\qn))$ is a Dirichlet form but it is not regular in general.
For more discussion of these two forms and the associated self-adjoint operators in our context, see  \cite{HKLW}. \\

The \textit{formal Laplacian} $\Delta$ can be introduced on $(X, w, \mu)$ as an analogue of the classical
Laplace-Beltrami operator on Riemannian manifolds as follows
$$(\Delta f)(x)=\frac{1}{\mu(x)}\sum_{y\in X}w(x,y)(f(x)-f(y)),$$
with domain
$F=\{f \in C(X) \mid \sum_{y\in V}w(x,y)|f(y)|<\infty\mbox{ for all }x\in X\}.$
Taking into account $\sum_{y}w(x,y)<\infty$, $x\in X$, the operator $\Delta$ is defined pointwise.
It is easy to see that $F$ is stable under multiplication by bounded functions on $X$.  
It can be shown that the self-adjoint operator $L$ with domain $D(L)$ corresponding to $Q$ is non-negative and is a restriction of $\Delta$, see \cite[Theorem~9]{KL}.   That is,
\begin{align*}
    Lu=\Delta u,\qquad u\in D(L).
\end{align*}

\subsection{Main results}\label{ss:main}

As discussed in the introduction, it is a classical result that the Laplacian on a weighted manifold is essentially self-adjoint
if all geodesic balls are relatively compact which is equivalent to the manifold being metrically complete (see, for example, Theorem~11.5 in \cite{GRIBOOK}).   Here we present some counterparts for weighted graphs.

We define the \textit{combinatorial neighborhood} $\n(K)$ of a subset $K$ of $X$ by
\[\n(K)=\{x\in X \mid x\in K \text{~~or there exists~~} y\in K \text{~~such that~~} x\sim y \}.\]
Note that the combinatorial neighborhood is not a topological notion and can be understood as the distance one ball about $K$ with respect to the natural graph distance.

\begin{theorem}
  \label{thm-main}{Let $(X,w,\mu)$ be a weighted graph and let $d$ be an intrinsic pseudo metric.  If the weighted degree function $\Deg$ is bounded on the combinatorial neighborhood of each distance ball, then
\begin{align*}
D(Q)&=D(\qn), \\
D(L)&=\{u\in L^{2}(X,\mu)\cap F\mid \Delta u\in L^{2}(X,\mu)\}.
\end{align*}
In particular, if additionally $\Delta C_{c}(X)\subseteq L^{2}(X,\mu)$, then $L_c = \Delta \vert_{C_c(X)}$ is essentially self-adjoint.}
\end{theorem}

\begin{rem}
(a) The result on essential self-adjointness is sharp by Example~\ref{Ex;2.1} in Section \ref{s:examples}.\\
(b) Let us note that the theorem does not assume that $\Deg$ is bounded on $X$. This would imply that $Q$ is bounded and the statements become trivial. \\
(c)  The condition $\Delta C_{c}(X)\subseteq L^{2}(X,\mu)$ holds if and only if $w(x, \cdot) / \mu(\cdot) \in L^2(X,\mu)$ for all $x \in X$, see Proposition 3.3 in \cite{KL}.  In particular, this always holds in the locally finite case.
\end{rem}

If a pseudo metric has finite jump size, the combinatorial neighborhood of a distance ball is contained in another distance ball. This yields the following immediate consequence.

\begin{cor}  If $(X,w,\mu)$ is a weighted graph and $d$ an intrinsic pseudo metric  with finite jump size such that  each distance ball is finite,
then the statements of Theorem~\ref{thm-main} hold.
\end{cor}

Note that finite balls and finite jump size imply that the graph is locally finite.
In this case, path pseudo metrics are metrics and the analogy to Riemannian manifolds becomes even more obvious as can be seen below.

Recall that an extension of $L_c= \Delta \vert_{C_c(X)}$ is said to be \textit{Markovian}, if the form associated to it is a Dirichlet form.  In particular, the operators associated to  $Q$ and $\qn$ are Markovian and, in the locally finite case, having a unique Markovian extension is equivalent to $D(Q) = D(\qn)$ \cite[Theorem 5.2]{HKLW}.

\begin{theorem}\label{thm-locfinite}
If $(X,w,\mu)$ is a locally finite weighted graph and  $d=d_\si$ an intrinsic path metric  such that $(X,d)$ is metrically complete, then
\[ D(L)=\{u\in L^{2}(X,\mu) \mid \Delta u\in L^{2}(X,\mu)\}, \]
$L_c = \Delta \vert_{C_c(X)}$ is essentially self-adjoint and has a unique Markovian extension.
\end{theorem}

\smallskip

Next, we turn to the metrically incomplete case where we will prove an analogue to results found in \cite{MASAMUNE}, see Theorem~\ref{thm-locfinite2} below.  In order to avoid some topological issues when defining the capacity, we now assume that all graphs are locally finite and that we only deal with path metrics. Note that by (a) of Lemma \ref{l:locfinite}, the topology induced by a path metric is discrete in the locally finite case.

For a set $U\subseteq X$ define the \textit{capacity} of $U$ by
\begin{align*}
    \mathrm{Cap}(U):=\inf\{\|u\|_{\widetilde{Q}}\mid u\in D(Q^{\max}),\; 1_{U}\leq u\},
\end{align*}
where $1_{U}$ is the characteristic function of $U$ and $\inf\emptyset=\infty$. For a path metric $d$ on $X$ we let $(\ov X,\ov d)$ be the metric completion.
We define the \textit{Cauchy boundary} $\cb$ of $X$ to be the difference between
$\overline X$ and $X$:
\[ \cb := \overline X \setminus X. \]
Clearly, $X$ is metrically complete if and only if $\cb$ is empty. For $A\subseteq\overline X$ define
\begin{align*}
    \mathrm{Cap}(A):=\inf\{\mathrm{Cap} (O\cap X)\mid A\subseteq O \mbox{ with } O \subseteq \ov X\,\mbox{open}\}.
\end{align*}
Note that, for $U \subseteq X$, the definitions of capacity agree due to the local finiteness and the use of path metrics.
We say that $\cb$ is \textit{polar} if $\mathrm{Cap}(\cb)=0$.\medskip

\begin{theorem}\label{thm-locfinite2}
Let $(X,w,\mu)$ be a locally finite weighted graph and  let $d=d_{\sigma}$ be a strongly intrinsic path metric.
If $(X,d)$ is not metrically complete and ${\rm Cap}(\cb) < \infty$, then
the Cauchy boundary is polar if and only if
\[D(Q)=D(\qn), \]
that is, if and only if $L_c = \Delta \vert_{C_c(X)}$ has a unique Markovian extension.
\end{theorem}

Note that the other consequences of Theorem~\ref{thm-main} do not necessarily follow if $(X,d)$ is metrically incomplete with polar boundary.  Example~\ref{Ex;2.1} in Section~\ref{s:examples} contains a weighted graph with polar boundary but where $L_{c}$ has non-Markovian self-adjoint extensions.  Also, in the case when the Cauchy boundary has infinite capacity, the Laplacian may be essentially self-adjoint or may not have a unique Markovian extension, see Examples~\ref{Ex;2.2}  and~\ref{Ex;2.3}.

\smallskip

Next we turn to a criterion which connects polarity of the boundary to co-dimension of the boundary.
The \textit{upper Minkowski codimension} $\cd(\cb)$ of $\cb$ is defined as
\[ \cd(\cb) = \limsup_{r \to 0} \frac{\ln \mu( B_r(\cb))} {\ln r}, \]
where $ B_r(\cb)=\{x\in X\mid \inf_{b\in \cb}\overline d(x,b)\leq r\}$. 
The upper Minkowski codimension (or box counting codimension) is one of the most
studied fractal dimensions. The relationships between various dimensions in the classical setting can be found in \cite{Fal}.

\medskip

\begin{theorem}\label{thm-Minkowski} Let $(X,w,\mu)$ be a locally finite weighted graph and let $d=d_\si$ be an intrinsic path metric.
If $\cd(\cb) > 2$, then $\cb$ is polar.
\end{theorem}

In Examples~\ref{Ex;2.4}, \ref{Ex;2.5}, and \ref{Ex;2.6} we show that $\cb$ can be polar or non-polar if $\cd(\cb)\leq2.$  See the recent paper \cite{GM} for some related statements in the case of manifolds.


\section{Uniqueness of solutions} \label{US}
Let $(X,w,\mu)$ be a weighted graph and let $d$ be an intrinsic pseudo metric.
 In this section we will show that under the assumption that $\mathrm{Deg}$ is bounded on the combinatorial neighborhood of each distance ball, there are no non-trivial $L^2$ solutions to $(\Delta+\lm)u=0$ for $\lm>0$. From this fact we can infer Theorems~\ref{thm-main} and ~\ref{thm-locfinite}.


\subsection{Leibniz rule and Green's formula}
The following auxiliary lemmas are well known in various other situations, see \cite{FLW,HK,HKLW,Jor,JP,KL,KL2}. However, they do not hold on graphs without further assumptions. Here, we prove them under the assumption that the weighted vertex degree is bounded on certain subsets.

\begin{lem}\label{l:fubini}
Let $B\subseteq X$ be such that $\Deg$ is bounded on $B$. Then, for all $u,v\in L^{2}(X,\mu)$,
\begin{align*}
\sum_{x,y\in B}w(x,y)|u(x)v(x)|<\infty\quad\mbox{ and } \quad \sum_{x,y\in B} w(x,y) | u(x)v(y)|<\infty.
\end{align*}
\end{lem}

\begin{proof}
Let $f=\max\{\vert u\vert,\vert v\vert\}$. Clearly, $f\in L^{2}(X,\mu)$. We estimate 
\begin{align*}
\sum_{x,y\in B}
w(x,y)|u(x)||v(x)|& \leq \sum_{x\in B}
f^2(x)\sum_{y\in B}w(x,y)
\leq \sum_{x\in B}f^2(x)\mu(x)\Deg(x)\\
&\leq\sup_{y\in B}\Deg(y)\|f\|^{2}<
\infty,
\end{align*}
since $\Deg$ is bounded on $B$ by assumption.
On the other hand, by the Cauchy Schwarz inequality and the above, we obtain
\begin{align*}
\sum_{x,y\in B}
w(x,y)|u(x)||v(y)|\leq
\Big(\sum_{x,y\in B}
w(x,y)u^2(x)\Big)^{\frac{1}{2}}\Big(\sum_{x,y\in B}
w(x,y)v^2(x)\Big)^{\frac{1}{2}}<\infty.
\end{align*}
\end{proof}

The following is an integrated Leibniz rule (see also  \cite[Theorem~3.7]{FLW}).
\begin{lem}\label{l:leibniz}(Leibniz rule) Let $U\subseteq X$ be such that  $\Deg$ is bounded on $\n(U)$. For all $f\in L^{\infty}(X)$ with $\supp f\subseteq U$ and $g,h\in L^{2}(X,\mu)$
\begin{align*}
\sum_{X}(\nabla (fg)\cdot{\nabla h})=\sum_{X}f(\nabla g\cdot{\nabla h})+\sum_{X}g(\nabla f\cdot{\nabla h}).
\end{align*}
\end{lem}
\begin{proof} By means of Lemma~\ref{l:fubini} and the fact that $\supp f\subseteq U$, it is not hard to see that $(fg,h),(f,h)\in D_{loc}(x)$ for $x\in n(U)$, and that $(g,h)\in D_{loc}(x)$ for $x\in U$. Moreover, all of the sums over $X$ above are, in fact, sums over $\n(U)$ and the first sum on the right hand side is over $U$. Hence,  by basic estimates, Lemma~\ref{l:fubini} and the fact $\supp f\subseteq U$, all of the sums above converge absolutely. Therefore, the statement follows by the simple algebraic manipulation $fg(x)-fg(y)=f(x)(g(x)-g(y))+g(y)(f(x)-f(y))$, $x,y\in X$.
\end{proof}

The following  Green's formula is a variant of \cite[Lemma~4.7]{HK}.

\begin{lem}\label{l:partial}(Green's formula) {Assume that $\mathrm{Deg}$ is bounded on $\n(U)$ for some set $U\subseteq X$. Then, for all $u,v\in L^{2}(X,\mu)\cap F$ with $\supp v\subseteq U$}
\begin{align*}
\sum_{ X}(\Delta u){v} \mu =\sum_{X} u {(\Delta v)}\mu=\frac{1}{2}\sum_{X} (\nabla u\cdot{\nabla v}).
\end{align*}
\end{lem}
\begin{proof}
Since $w(x,y)=0$ whenever $x\in U, y\in X\backslash\n(U)$, the statement follows by simple algebraic manipulations,  Lemma~\ref{l:fubini}  and Fubini's theorem.
\end{proof}

\subsection{A Caccioppoli-type inequality}
The key estimate for the proof of triviality of $L^{2}$ solutions to $(\Delta + \lm)u = 0$ for $\lm>0$ is the following Caccioppoli-type inequality. See \cite[Theorem~11.1]{FLW} for a similar result for general Dirichlet forms.

\begin{lem}\label{l:key}(Caccioppoli-type inequality) Let $u\in L^2 (X, \mu)\cap F$, $U\subseteq X$
and assume that $\Deg$ is bounded on $\n(U)$.
Then, for all $v\in L^{\infty}(X)$ with $\supp v\subseteq U$,
\begin{align*} -\sum_{X}(\Delta u) {uv^{2}}\mu\leq\frac{1}{2}\sum_{X} u^2|\nabla v|^{2}.
\end{align*}
\end{lem}
\begin{proof}
By Lemma~\ref{l:leibniz} and Lemma~\ref{l:partial} we have
\begin{align*}\label{e:partial}
\sum_{X}(\Delta u){uv^{2}}\mu&=
\frac{1}{2}\sum_{X}(\nabla u\cdot{\nabla (uv^{2})})= \frac{1}{2} \sum_{X}{v^{2}}|\nabla u|^{2} +\frac{1}{2}\sum_{X}{u}(\nabla u\cdot{\nabla v^{2} }).
\end{align*}
We focus on the second sum on the right hand side.
Since a geometric mean can be estimated by its corresponding arithmetic mean, we have $|ab|\leq \frac{\delta}{2} a^{2}+\frac{1}{2\delta}b^{2}$ for $a,b\in\R$ and $\delta>0$.
We use this estimate with $a=(u(x)-u(y))(v(x)+v(y))$, $b=u(x)(v(x)-v(y))$  and $\delta=1/2$ for the terms in the second sum on the right hand side above
\begin{align*}
|u(x)(u(x)-u(y))(v^{2}(x)-v^{2}(y))| &\leq\frac{1}{4}(u(x)-u(y))^{2}(v(x)+v(y))^{2} +u^2(x)(v(x)-v(y))^{2}\\
&\leq\frac{1}{2}(u(x)-u(y))^{2}(v^2(x)+v^2(y)) +u^2(x)(v(x)-v(y))^{2}.
\end{align*}
Multiplying by $w(x,y)$ and summing over $x,y\in X$ yields $$-\frac{1}{2}\sum_{X}{u}(\nabla u\cdot{\nabla v^{2} })\leq \frac{1}{2} \sum_{X}v^{2}|\nabla u|^{2} +\frac{1}{2} \sum_{X}u^{2}|\nabla v|^{2}.$$
The assertion now follows from the equality in the beginning of the proof.
\end{proof}

\subsection{Uniqueness of solutions}
The following  is an analogue of \cite[Lemma 11.6]{GRIBOOK}.

\begin{pro}\label{pro-uniqueness}
Assume that the weighted degree function $\mathrm{Deg}$
is bounded on the combinatorial neighborhood of each distance ball.
Then, for all $\lambda>0$, the
equation
$$(\Delta+\lambda)u=0,$$
has only the trivial solution in $L^2 (X, \mu) \cap F$.
\end{pro}
\begin{proof}
Let $0 \leq r<R$, fix $x_0 \in X$, and consider the cut-off function $\eta=\eta_{R,r}:X\to\R$ given by
\begin{equation*}
\eta(x)=\left(\frac{R-d(x, x_0)}{R-r}\right)_+ \wedge 1.
\end{equation*}
Note that $0\leq\eta\le 1$, $\eta\vert_{B_{r}}=1$ and $\eta\vert_{X\setminus B_{R}}=0$. Moreover, $\eta$ is Lipshitz continuous with Lipshitz constant $\frac{1}{R-r}$ (of course, with respect to the intrinsic pseudo metric $d$). This immediately implies, as $d$ is an intrinsic
pseudo metric, that
\begin{align*}
|\nabla \eta|^{2}(x)\leq \frac{1}{(R-r)^{2}}\sum_{y\in X}w(x,y)d^{2}(x,y)\leq\frac{\mu(x)}{(R-r)^{2}},\qquad x\in X.
\end{align*}
Fix $\lambda>0$ and assume that $u\in L^2 (X, \mu)\cap F$ is a
solution to the equation $(\Delta+\lambda)u=0.$
Then, we have by Lemma~\ref{l:key} and the estimate on $|\nabla \eta|^{2}$ above, that
\begin{align*}
\lambda\aV{u1_{B_{r}}}^{2}
\leq \lambda\aV{u\eta}^{2}=-\sum_{X}(\Delta u)u\eta^{2} \mu \leq \frac{1}{2}\sum_{X} u^2|\nabla \eta|^{2} \leq \frac{1}{2(R-r)^{2}}\aV{u}^{2}.
\end{align*}
Letting $R\rightarrow \infty$,  we see that  $u\equiv 0$ on $B_r$. Since $r$ is  chosen arbitrarily, $u\equiv 0$ on $X$.
\end{proof}

\begin{rem}It would be interesting to prove a similar statement as Proposition~\ref{pro-uniqueness} for $L^{p}$.
\end{rem}

\subsection{Proofs of Theorems~\ref{thm-main} and~\ref{thm-locfinite}}

The proof of  Theorem~\ref{thm-main} follows by standard techniques used in \cite{KL} and \cite{HKLW}.

\begin{proof}[Proof of Theorem~\ref{thm-main}]  By \cite[Corollary~4.3]{HKLW}, $D(Q)=D(\qn)$ is equivalent to the non-existence of non-trivial solutions to $(\Delta+\lambda)u=0$ for $\lambda>0$ in $D(\qn)$. Thus,  $D(Q)=D(Q^{\max})$ follows from Proposition~\ref{pro-uniqueness}.

Define ${D}_{\max}
=\{u\in L^2(X,\mu)\cap F\mid \Delta u \in L^2(X,\mu)\}.$
By Theorem~9 in \cite{KL}, $L$ is a restriction of $\Delta$ which implies that $D(L)\subseteq {D}_{\max}
$. Letting $f\in D_{\max}
$ we see that $g:=(\Delta+\lambda)f\in L^{2}(X,\mu)$ for all $\lambda>0$, so that $u:=(L+\lambda)^{-1}g\in D(L)$. As $u$ solves the equation $(\Delta +\lambda)u=g$ (see Lemma~2.8 in \cite{KL}), we conclude that $f=u$ by the uniqueness of solutions, Proposition~\ref{pro-uniqueness} (as $f$ solves the equation by definition). Thus, $f\in D(L)$ and, therefore, $D(L)=D_{\max}
$.

Assuming $\Delta C_{c}(X)\subseteq L^{2}(X,\mu)$, essential self-adjointness is a rather immediate consequence of $L=L_c^{*}$. By Green's formula for functions in $v\in C_{c}(X)$ and $f\in F$ (see \cite[Lemma~4.7]{HK} or \cite[Proposition~3.3]{KL}), we have $\sum_{X}f (L_{c}v) \mu=\sum _{X}(\Delta f) v  \mu$ and thus $D(L_c^*) = D_{\max}$.
Hence, by what we have shown above, we have
$D(L_{c}^{*})=D(L)$ and, therefore, it follows that $L=L_{c}^{*}$.
\end{proof}

\begin{proof}[Proof of Theorem~\ref{thm-locfinite}]
As we assume local finiteness, it is clear that $\Delta C_c(X) \subseteq L^2(X,\mu)$.  Furthermore, by Theorem~\ref{t:locfinite}, the metric  completeness of $(X,d)$ implies that distance balls are finite.  Note that the combinatorial neighborhood of a finite set is again finite. Hence, $\Deg$ is bounded on the combinatorial neighborhood of each distance ball which implies the statements about essential self-adjointness and $D(Q) = D(\qn)$ by Theorem~\ref{thm-main}. As uniqueness of Markovian extensions  is equivalent to $D(Q) = D(\qn)$ in the locally finite case, see \cite[Theorem 4.2]{HKLW}, the second statement follows as well. 
\end{proof}

\section{Cauchy boundary and equilibrium potentials} \label{PM}
Let $(X,w,\mu)$ be a locally finite weighted graph and let $d$ be a path metric. Recall that $(\overline X,\overline d)$ denotes the metric completion of $(X,d)$ and $\partial_{C}X=\overline X\setminus X$ denotes the Cauchy boundary. In this section we prove Theorems~\ref{thm-locfinite2} and ~\ref{thm-Minkowski}.


\subsection{Existence of  equilibrium potentials}
The following is well known and follows directly from \cite[Lemma~2.1.1.]{FOTBOOK}.
\begin{lem}
If $\mathrm{Cap}(O)<\infty$ for an open set $O \subset \overline X$,
then there is a unique element $e \in D(\qn)$ such that  $0 \leq e \leq 1$, $e|_{O \cap X} \equiv 1$, and $\mathrm{Cap}(O)=\| e \|_{\widetilde Q}$.
\end{lem}
\begin{proof}  From \cite[Lemma~2.1.1.]{FOTBOOK} it follows that for any $U\subseteq X$ there is such an $e\in D(Q^{\max})$ (as we consider $X$ equipped with the discrete topology).  Note that in \cite[Lemma~2.1.1.]{FOTBOOK}  regularity of the form is a standing assumption but this is not needed for the proof.   Now, for an open set $O\subseteq \ov X$ the equality $\mathrm{Cap}(O)=\mathrm{Cap}(O\cap X)$ follows from the definition (by taking $A=O$). Hence, we let $e$ for $O$ be the corresponding $e$ for $O\cap X$. 
\end{proof}

We call such an $e$ the \textit{equilibrium potential} associated to $O$.


\subsection{The boundary alternative}\label{s:boundaryalternative}
The following lemma shows that if the minimal and the maximal forms agree, then the capacity of any subset of the boundary is either zero or infinite.

\begin{lemma}\label{l:boundary} Let $A\subseteq \partial_{C}X$.
 If $D(Q)=D(\qn)$, then either $\mathrm{Cap}(A)=\infty$ or $\mathrm{Cap}(A)=0$.
\end{lemma}
\begin{proof} Assume that $D(Q) = D(Q^{\max})$ and $A$ has finite capacity. Then, there exists an open set $O \subseteq \ov X$ such that $A\subseteq O$ and $\cp(O)< \infty$.  Let $e$ be the equilibrium potential associated to $O$. Since $D(Q) = D(Q^{\max})$ there exists a sequence of functions  $e_n$  in $C_c(X)$ converging to $e$ as $n \to \infty$ in the $\aV{\cdot}_{\widetilde{Q}}$ norm.
Clearly, $(e-e_n)_+ \wedge 1$ belongs to $D(Q^{\max})$ and equals $1$ on $O_n\cap X$, where $O_n$ is a neighborhood of $A$ in $\overline{X}$.  Therefore,
\[ {\rm Cap}(A) \le \liminf_{n\to\infty} {\rm Cap}(O_n) \le \lim_{n\to\infty} \| (e-e_n)_+ \wedge 1 \|_{\widetilde{Q}}\leq \lim_{n\to\infty}\| e-e_n \|_{\widetilde{Q}}= 0.    \]
\end{proof}


\subsection{Approximation by equilibrium potentials}\label{s:boundaryalternative}

Next, we show that every bounded function in $D(Q^{\max})$ can be approximated via  equilibrium potentials if the boundary has capacity zero.

\begin{lemma}\label{l:equpot} Assume that $\partial_{C}X$ is polar and let $e_{n}$ be the equilibrium potentials associated to open sets $O_{n}\subseteq \overline X$ with $\cb \subseteq O_n$ and $\mathrm{Cap}(O_{n})\to0$ as $n \to \infty$. Then $\|u-(1-e_{n})u\|_{\ow Q}\to0$ as $n\to\infty$ for all $u\in D(Q^{\max}) \cap L^\infty(X)$.
\end{lemma}
\begin{proof}  Note that $u-(1-e_{n})u=e_{n}u$ and $\|e_{n}u\|\leq \|u\|_{\infty}\|e_{n}\|\to 0$ as $n\to\infty$. Moreover, we have
\begin{align*}
\ow{Q}(e_{n}u) &= \frac{1}{2} \sum_{X} | \nabla (e_n u) |^2 \leq \sum_{X} e_n^2 | \nabla u|^2 + \sum_{X} u^2 | \nabla e_n |^2 \\
& \leq \sum_{X} e_n^2 | \nabla u|^2 + 2 \aV{u}_{\infty}^2 \ow{Q}(e_n) \to 0\quad\mbox{as $n \to \infty$}
\end{align*}
by noting that $e_n(x) \to 0$ for all $x$ and applying the Lebesgue dominated convergence theorem.
\end{proof}


\subsection{Restriction to complete subgraphs}

In the next lemma we show that bounded functions in $D(Q^{\max})$ that are zero close to the boundary can be approximated by finitely supported functions.
We show this by restricting our attention to complete subgraphs. In order for the restriction of an intrinsic  metric to be intrinsic, we need to assume that the metric is strongly intrinsic.

\begin{lemma}\label{l:subgraphs} Assume that  $d=d_{\si}$ is a strongly intrinsic path metric. Let $O\subseteq \overline X$ be open with $\partial_{C}X\subseteq O$, $\mathrm{Cap}(O)<\infty$ and let $e$ be the equilibrium potential associated to $O$. Then, $C_{c}(X)$ is dense in $(1-e)(D(Q^{\max}) \cap L^\infty(X)) =\{(1-e)u\mid u\in D(Q^{\max}) \cap L^\infty(X)\}$ with respect to $\aV{\cdot}_{\ow{Q}}.$
\end{lemma}

\begin{proof}
Let $Y= X \setminus O$, $\mu_{Y}$ be the restriction of $\mu$ to $Y$ and $\si_{Y}$ and $w_{Y}$ be the restrictions of $\si$ and $w$ to $Y\times Y$. We first assume that $(Y,w_{Y},\mu_{Y})$ is connected. From this it follows that $d_{Y}=d_{\si_{Y}}$ is a strongly intrinsic  path metric on $(Y,w_{Y},\mu_{Y})$ with $d_{Y}\geq d$.

\textit{Claim:} $(Y,d_{Y}) $ is metrically complete. \\
\textit{Proof of the claim:} If $(x_{n})$ is a Cauchy sequence in $Y$, then, by $d_{Y}\ge d$, it is a Cauchy sequence in $X$ and has a limit point in $\overline X$. However, as $Y=X\setminus O$, the limit point is not in $\partial_{C} X$. As $(X,d_{\si})$  is a discrete metric space, see Lemma~\ref{l:locfinite}~(a), $(x_{n})$ must be eventually constant which proves the claim.

\smallskip

For $R>0$ and fixed $x_{0}\in Y$, let $\eta_{R}:Y\to[0,1]$ given by
\begin{align*}
    \eta_{R}(x)=\left(\frac{2R - d_Y(x,x_{0})}{R}\right)_{+} \wedge 1.
\end{align*}
By completeness, $B_{2R}$ in $(Y,d_{Y})$ is finite, see Theorem A.1 in Appendix~\ref{PML}, and it follows that $\eta_{R}(1-e)D(Q^{\max})\subseteq C_c(X)$.

Let $\ow\nabla$ be the generalized gradient for $(Y,w_{Y},\mu_{Y})$.
Let $v\in(1-e) (D(Q^{\max})\cap L^{\infty}(X))$ and set $g_{R}=v-\eta_{R}v=(1-\eta_R)v$.
Now,
\begin{align*}
\ow Q(g_{R})=\frac{1}{2}\sum_{Y}|\ow\nabla g_{R}|^{2}+\sum_{x\in Y} \sum_{y\in O\cap X}w(x,y)g_{R}^{2}(x).
\end{align*}
For the first term we get, using that $d_{Y}$ is intrinsic and, therefore, that $|\ow\nabla \eta_{R}|^{2}\le \mu_Y/R^{2}$,
\begin{align*}
\frac{1}{2} \sum_{Y}|\ow\nabla g_{R}|^{2}&\leq\sum_{Y}v^{2}|\ow\nabla \eta_{R}|^{2}+\sum_{Y}(1-\eta_{R})^{2}|\ow\nabla v|^{2}\leq \frac{1}{R^{2}}\|v\|^{2}+\sum_{Y}(1-\eta_{R})^{2}|\ow\nabla v|^{2}\to0,
\end{align*}
as $R\to\infty$. For the second term let $u\in (D(Q^{\max})\cap L^{\infty}(X))$ such that $v=(1-e)u$. Then,
\begin{align*}
\sum_{x\in Y} \sum_{y\in O\cap X}w(x,y)g_{R}(x)^{2} \leq \|u\|_{\infty}^{2} \sum_{x\in Y} (1-\eta_{R}(x))^{2}\sum_{y\in O\cap X}w(x,y)(1-e(x))^{2}\to0
\end{align*}
as $R \to \infty$ by the Lebesgue dominated convergence theorem.  This follows, since $\eta_{R}\to1$ pointwise as $R\to\infty$ and ${\sum_{Y}\sum_{O\cap X}}w(x,y)(1-e(x))^{2}\leq \ow Q(e)\leq \mathrm{Cap}(O)^{2}< \infty$ (as $e(y)=1$ for $y\in O\cap X$).
Moreover, as $g_{R}$ converges pointwise to zero it also converges to zero in $L^{2}$.

In the case where $Y$ is not connected there are at most countably many connected components $Y_{i}$, $i\geq0$. For $v\in(1-e) (D(Q^{\max}) \cap L^\infty(X))$ let $v_{i}=v\vert_{Y_{i}}$. Since $\ow Q(v)=\sum_{i\geq 0} \ow Q(v_{i})$, the statement follows by a diagonal sequence argument.
\end{proof}


\subsection{Proof of Theorem~\ref{thm-locfinite2}}
\begin{proof}[Proof of Theorem~\ref{thm-locfinite2}]
Since we assume local finiteness, $L_c$ having a unique Markovian extension is equivalent to $D(Q)=D(\qn)$ by Theorem 5.2 in \cite{HKLW}.

If $D(Q)=D(\qn)$, then the assumption $\mathrm{Cap}(\partial_{C}X)<\infty$ implies   $\mathrm{Cap}(\partial_{C}X)=0$ by Lemma~\ref{l:boundary}.

If, on the other hand,  $\partial_{C}X$ has zero capacity, then $C_{c}(X)$ is  dense in $D(Q^{\max}) \cap L^\infty(X)$ with respect to $\aV{\cdot}_{\ow{Q}}$ by Lemmas~\ref{l:equpot} and~\ref{l:subgraphs}.   But  $D(Q^{\max}) \cap L^\infty(X)$ is dense in $D(\qn)$ with respect to $\aV{\cdot}_{\ow{Q}}$ since if  $u \in D(\qn)$, then $u_n = (u \vee -n) \wedge n$ converges to $u$ in $\aV{\cdot}_{\ow{Q}}$ as $n \to \infty$ (cf. \cite[Theorem 1.4.2 (iii)]{FOTBOOK}).  This implies $D(Q)=D(\qn)$.
\end{proof}


\subsection{Proof of Theorem~\ref{thm-Minkowski}}
\begin{proof}[Proof of Theorem~\ref{thm-Minkowski}]
As $\cd(\cb)>2$, there exists an $\eps>0$ and a sequence $r_n \to 0$ as $n \to \infty$ such that
\[ \mu(B_{r_n}(\cb)) < r_n^{2+\eps}. \]
For $R>0$ and $x \in X$, let
\[ \eta_R(x) = \left( \frac{2R - \ov{d}(x, \cb)}{R} \right)_+ \wedge 1. \]
In particular, $0 \leq \eta_R \leq 1$, $\eta_R |_{B_R(\cb)} \equiv 1$ and $\eta_R |_{X \setminus B_{2R}(\cb)} \equiv 0$.  It follows that
\[ \aV{\eta_R}^2 \leq \mu(B_{2R}(\cb)) \]
and
\begin{align*}
\ow{Q}(\eta_R) &\leq  \sum_{x\in B_{2R}(\cb)}\sum_{y \in X} w(x,y) \left( \frac{\ov{d}(x,\cb)}{R} - \frac{\ov{d}(y,\cb)}{R} \right)^2 \\
& \leq \frac{1}{R^2} \sum_{x\in B_{2R}(\cb)}\sum_{y \in X} w(x,y) d(x,y)^2 \leq \frac{\mu(B_{2R}(\cb))}{R^2}
\end{align*}
since $d$ is intrinsic.

Applying the above with $r_n/2$ in place of $R$, it follows that
\begin{align*}
\cp(\cb) &\leq \aV{\eta_{r_n/2}}_{\ow Q} \leq \left( \mu(B_{r_n}(\cb)) + \frac{4}{r_n^2} \mu(B_{r_n}(\cb)) \right)^{\frac{1}{2}} \\
&  \leq \left( r_n^{2+\eps} + 4 r_n^\eps \right)^{\frac{1}{2}} \to 0 \quad \mbox{as } n \to \infty.
\end{align*}
\end{proof}


\section{(Counter-)examples} \label{s:examples}
Here we present  the examples mentioned in Section~\ref{ss:main}.
In particular, we show that Markov uniqueness does not imply essential self-adjointness, that no conclusion can be drawn concerning uniqueness in the infinite capacity case and that the boundary can be polar or non-polar for any upper Minkowski codimension less than or equal to 2.

As we often use a graph with $X=\N_0$ and $x \sim y$ if and only if $|x-y|=1$ we make several preliminary observations concerning graphs of this type with a given path metric $d$.  First, in this case,
\[ \cb \not = \emptyset \quad \textrm{ if and only if } \quad l(X):=\sum_{x=0}^\infty d(x,x+1) < \infty,\]
see Theorem~\ref{t:locfinite} in Appendix~\ref{PML}.
Second, if $\cb \not = \emptyset$, then
\[ \cp(\cb)<\infty \quad \textrm{ if and only if } \quad \mu(X)<\infty. \]
This can be seen as follows: if $\mu(X)=\infty$, then every neighborhood of the boundary must have infinite measure so that $\cp(\cb)=\infty$.  If $\mu(X)<\infty$, then $1 \in D(\qn)$ which implies that $\cp(\cb)\leq \aV{1}_{\ow Q} = \mu(X) < \infty$.
These two observations will be used repeatedly below.

\begin{eg}[Polar Cauchy boundary (and consequently $D(Q) = D(\qn)$) but no essential self\-adjointness] \label{Ex;2.1}
Let $X = \mathbb Z$ with $w(x,y) = 1$ if $|x-y|=1$ and 0 otherwise.
The strongly intrinsic  path metric $d=d_{\sigma_0}$ introduced in Example~\ref{adapted} satisfies  $d(x,x+1) = \min \{ \sqrt{{\mu(x)}/{2}},\sqrt{{\mu(x+1)}/{2}},\ 1 \}$ for an arbitrary measure $\mu$.
Therefore, if the measure is chosen so that it satisfies $\sum_{x=-\infty}^\infty x^2 \sqrt{\mu(x)} < \infty$, then $(X,d)$ is metrically incomplete, the Cauchy boundary consists of two points  and $h:x\mapsto x$ is in $L^{2}(X,\mu)$ which we will use later.

Define $e_n$ by
\[  e_n (x) :=  \left( {|x|}/{n}  -1 \right)_+ \wedge 1. \]
One checks that $e_n \in D(\qn)$ with
\[ \ow Q(e_n)= \sum_{x=-\infty}^\infty (e_{n}(x)-e_{n}(x+1))^2 =  2 n \frac{1}{n^{2}} \to 0 \]
and that $e_n \to 0$ in $L^2(X,\mu)$ as $n \to \infty$ by the Lebesgue dominated convergence theorem.
Thus, the Cauchy boundary of $X$ is polar and $L_c=\Delta \vert_{C_c(X)}$ has a unique Markovian extension.

On the other hand, the formal Laplacian $\Delta$ acts as  $\Delta f (x)= \frac{1}{\mu(x)}  ( f(x) - f(x-1) + f(x) - f(x+1))$. Clearly, $h (x) = x$ is harmonic, square integrable by the choice of $\mu$, and $h \notin  D(\qn)$.
This shows that $h  \in D(L_c^*) \setminus D(\qn)$, that is,
$L_c$ is not essentially self-adjoint.
\end{eg}


\begin{eg} [Cauchy boundary with infinite capacity and essential self-adjointness] \label{Ex;2.2}
Let $X = \N_0$ with $\mu(X) = \infty$ and $w$ symmetric such that $w(x,y)>0$ if and only if $|x-y| = 1$. By \cite[Theorem~6]{KL} the operator $\Delta \vert_{C_c(X)}$ is essentially selfadjoint. (This can be also seen directly  as there are no non-trivial solutions to $(\Delta + \lm)u=0$ in $L^2(X, \mu)$ for $\lm > 0$.  This follows as any positive solution to this equation must be increasing by a minimum principle, see also equation \eqref{increments} below.)

If $d=d_{\si_0}$ and $w$ and $\mu$ are chosen to satisfy
\[l(X)=\lim_{x\to\infty}d(0,x) \leq \sum_{x=0}^\infty \left( \frac{1}{\Deg(x)} \right)^{\frac{1}{2}} \leq \sum_{x=0}^\infty \left(\frac{\mu(x)}{w(x,x+1)} \right)^{\frac{1}{2}} < \infty, \]
then it follows that $(X,d)$ is not metrically complete and that the boundary consists of a single point.  Since $\mu(X) = \infty$, $\cp(\cb) = \infty$ as noted above.
\end{eg}


\begin{eg}[Cauchy boundary with finite positive capacity and consequently $D(Q) \not = D(\qn)$] \label{Ex;2.3a}
Let $X=\N_{0}$ with $\mu(X)<\infty$  and let $w$ be symmetric with $w(x,y)>0$ if and only if $|x-y|=1$ and satisfying
\[l(X)=\lim_{x\to\infty}d(0,x)\le \sum_{x=0}^\infty \left( \frac{\mu(x)}{w(x,x+1) } \right)^{\frac{1}{2}} < \infty \quad\mbox{ and }\quad \sum_{x=0}^\infty \frac{1}{w(x,x+1)} < \infty \]
where $d=d_{\si_0}$.
In particular, the Cauchy boundary $\cb$ of $X$ consists of one point and has finite capacity.

Recall that  $D(Q) \not = D(\qn)$ is equivalent to  $(\De + \lm)u = 0$ having a non-trivial solution in $D(\qn)$ for any $\lm>0$  \cite[Corollary 4.3]{HKLW}.  By \cite[Lemma 4.3]{KLW}, the equation $(\De + \lm)u = 0$ on $X$ translates to
\begin{equation} \label{increments}
 u(x+1) - u(x) = \frac{\lm}{w(x,x+1)} \sum_{y=0}^x u(y) \mu(y)
\end{equation}
from which it follows, see \cite[Lemma 5.4]{KLW},  that $u$ is bounded if and only if
\[ \sum_{x=0}^\infty \frac{\sum_{y=0}^x \mu(y)}{w(x,x+1)} < \infty. \]
As $\mu(X) < \infty$, this is equivalent to
\[ \sum_{x=0}^\infty \frac{1}{w(x,x+1)} < \infty. \]
Furthermore, as $\mu(X) < \infty$, $u \in L^\infty(X)$ implies that $u \in L^2(X, \mu)$.  It is also not difficult to see that $\ow Q(u) < \infty$ in this case as, by (\ref{increments}), we get that
\[ w(x,x+1) (u(x+1)-u(x))^2 \leq \frac{1}{w(x,x+1)} \left(\lm \mu(X)\aV{u}_{\infty} \right)^2. \]
Therefore, as $u \in D(\qn)$ is non-trivial, $D(Q) \not = D(\qn)$.
Finally, $\cp(\partial_{C}X) >0$ follows by combining $\cp(\partial_{C}X) < \infty$,  $D(Q) \not = D(\qn)$, and Theorem~\ref{thm-locfinite2}.
\end{eg}

\begin{eg}[Cauchy boundary with infinite capacity and $D(Q) \not = D(\qn)$ (and consequently no essential self-adjointness)] \label{Ex;2.3}
We consider $X=\Z$ with $X = X_- \cup X_+$ where $X_-= - \N_{0}$ with $w$ and $\mu$ chosen as in  Example~\ref{Ex;2.2} and $X_+ = \N_0$  with $w$ and $\mu$ chosen as in Example~\ref{Ex;2.3a}.
In particular, the Cauchy boundary $\cb$ of $X$ consists of two points, $p_L$ and $p_R$, and has infinite capacity as $\cp(p_{L})=\infty$ by  $\mu(X_{-})=\infty$. On the other hand, by Example~\ref{Ex;2.3a} we have $0<\cp(p_{R})<\infty$ which gives  $D(Q) \not = D(\qn)$ by Lemma~\ref{l:boundary}.
\end{eg}


\begin{eg}[Polar Cauchy boundary with upper Minkowski codimension~2]\label{Ex;2.4}

Let $X=\N_0$ with $w(x,y) =1/8$ if $|x-y|=1$ and 0 otherwise and $\mu(x) = 4^{-x}$.  Therefore, for $x>0$, $\Deg(x) = 4^{x-1}$ so that, with  $d= d_{\si_0}$, we get  $d(x,x+1) = 2^{-x}.$  Furthermore, by using the technique of Example~\ref{Ex;2.1}, we can show that the Cauchy boundary consists of a single point, $p_R$, and that $\cp(p_R) = 0$.
Let $r(x) := \ov{d}(x, p_R) = \sum_{y=x}^\infty 2^{-y} = 2^{-(x-1)}$
so that $\mu(B_{r(x)}(p_R)) = \sum_{y=x}^\infty 4^{-y} = 4^{-(x-1)}/3 =r(x)^{2}/3$.
Therefore,
\[ \frac{\ln \mu(B_{r(x)}(p_R)) }{ \ln r(x)} = \frac{2\ln r(x)-\ln3}{\ln r(x)}\to 2 \quad \mathrm{ as } \ x \to \infty. \]
\end{eg}


\begin{eg}[Non-polar Cauchy boundary with upper Minkowski codimension 2]\label{Ex;2.5}
Let $X = \N_0$, with $w$ symmetric, satisfying $w(x,x+1) = (x+1)^2$ and 0 otherwise with
\[  d(x,x+1) = \frac{1}{2^{x+2}} \qquad \textrm{ and } \qquad \mu(x) = \frac{(x+1)^2}{4^x}. \]
It is easy to check that this metric is intrinsic.
As $l(X) < \infty$, $\mu(X) < \infty$, and $\sum_{x=0}^\infty \frac{1}{w(x,x+1)} < \infty$ it follows that
\[ 0 < \cp(\cb) < \infty \]
by the reasoning of Example~\ref{Ex;2.3a}.   Therefore, $\cd(\cb) \leq 2$ by Theorem~\ref{thm-Minkowski}.
By definition,
\begin{align*}
 r(x) := \overline{d}(x, \cb) = \frac{1}{2^{x+1}}
\quad\mbox{and}\quad
 \mu(B_{r(x)}(\cb) ) = \sum_{y=x}^\infty \frac{(y+1)^2}{4^y}.
\end{align*}
 Now, for every $\beta>1/4$, there exists an $M$ such that  $(x+1) ^2 \leq (4\be)^x $
for all $x \geq M$.  Hence, for all $x \geq M$ and $1/4 < \be < 1$, we have
 $\mu(B_{r(x)}(\cb)) \leq \sum_{y=x}^\infty \be^{y}=\be^x ({1-\be} )^{-1} .$
Therefore, for all $x \geq M$,
\[ \frac{\ln \mu(B_{r(x)}(\cb))}{\ln r(x) } \geq \frac{\ln \left( \be^x (1-\be)^{-1} \right)}{ \ln 2^{-(x+1)}} \]
which implies that
$\cd(\cb) \geq -{\ln \be}/{\ln {2}}. $
As $\beta>1/4$ was chosen arbitrarily, it follows that $\cd(\cb) \geq 2$ yielding that $ \cd(\cb) = 2$.
\end{eg}


\begin{eg}[Upper Minkowski codimension between 0 and 2 with polar and non-polar Cauchy boundary]\label{Ex;2.6}
Let $X = \N_0$ with $w(x,y)>0$ if and only if $|x-y|=1.$
Let, for $\al \in \mathbb{R}$,
\[ d(x,x+1) = \frac{1}{2^{\al(x+1)}} \qquad \textrm{ and  } \qquad \mu(x) = \frac{1}{2^{(2\al -1 )x}}. \]
Then, $l(X)<\infty$ for $\al>0$ and $\mu(X)<\infty$ for $\al>1/2$. Thus, $ \cp(\cb) < \infty$ for $\al>1/2$.
Now,
\[ r(x) := \overline{d}(x, \cb) = \sum_{y=x}^\infty  \frac{1}{2^{\al(y+1)}} =  \left( \frac{1}{2^\al -1} \right) \frac{1}{2^{\al x}} \]
and
\[ \mu\left( B_{r(x)}(\cb) \right) = \sum_{y=x}^\infty  \frac{1}{2^{(2\al -1 )y}} = \left( \frac{1}{2^{2\al -1} -1} \right) \frac{1}{2^{(2\al -1)(x-1) }} \]
so that
\[ \cd(\cb) = 2 - \frac{1}{\al} .\]
We now specify two choices of weights $w$: \\
\textbf{Case 1 - polar Cauchy boundary:}
Let $ w(x,x+1)=1$ for all  $x \in \N_0.$ Clearly, $d$ is intrinsic for all $\al > 0$.  Furthermore, if $\cp(\cb) < \infty$, then $\cp(\cb)=0$ as in Example~\ref{Ex;2.1}.  Hence, there exist examples of graphs with polar Cauchy boundary such that $0 < \cd(\cb) < 2$.

\noindent \textbf{Case 2 - non-polar Cauchy boundary:}
Let  $w(x,x+1) = 2^x$ for all  $x \in \N_0$.
It is easy to see that $d$ is intrinsic for all $\al \geq \frac{1}{2}$.  Furthermore, since  $\sum_{x=0}^\infty \frac{1}{w(x,x+1)} < \infty$
one can show that $\cp(\cb)>0$ as in Example~\ref{Ex;2.3a}. Consequently, there exists a family of graphs with non-polar Cauchy boundary such that $0 < \cd(\cb) < 2$.
\end{eg}

\appendix

\section{A Hopf-Rinow type theorem}
\label{PML}

Let $(X,w,\mu)$ be a weighted graph.

A metric space $(X,d)$ is said to be \textit{metrically complete} if every Cauchy sequence converges to an element in $X$.
A path $(x_{n})$ (finite or infinite) is called a \textit{geodesic} with respect to a  path metric $d=d_{\si}$ if $d(x_{0},x_{n})=l_{\si}((x_{0},\ldots,x_{n}))$ for all $n\geq0$. A weighted graph $(X,w,\mu)$ with a path metric $d=d_{\si}$ is said to be \textit{geodesically complete} if  all infinite geodesics have infinite lengths, i.e.,  $l_{\si}\left((x_{k})\right)=\lim_{n} l_{\si}((x_{0},\ldots,x_{n}))=\infty$ for all infinite geodesics $(x_{k})$.

We prove the following Hopf-Rinow type theorem.

\begin{thm}\label{t:locfinite}
Let $(X,w,\mu)$ be a locally finite weighted graph and $d$ be a path pseudo metric. Then, $(X,d)$ is a discrete metric space. Moreover, the following are equivalent:
\begin{itemize}
\item [(i)] $(X,d)$  is metrically complete.
\item [(ii)]  $(X,d)$  is geodesically complete.
\item [(iii)] Every distance ball is finite.
\item [(iv)] Every bounded and closed set is compact.
\end{itemize}
In particular, if $(X,d)$ is complete, then for all $x,y\in X$ there is a path $(x_{0},\ldots,x_{n})$ connecting $x$ and $y$ such that $d_{\si}(x,y)=l_{\si}((x_{0},\ldots,x_{n}))$.
\end{thm}

\begin{rem}(a) Anytime a path pseudo metric $d$ induces the discrete topology on $X$ the following implications hold: (iii)$\Leftrightarrow$(iv)$\Rightarrow$(i)$\Rightarrow$(ii).  This is the case if and only if $\inf_{y \sim x} \sigma(x,y) >0$ for all $x \in X$.  In fact, (iv)$\Rightarrow$(i) holds for general metric spaces.
The stronger assumption of local finiteness is needed for the implications (ii)$\Rightarrow$(i), (i)$\Rightarrow$(iii) (or (iv)) and (ii)$\Rightarrow$(iii) (or (iv)).  See Example \ref{Ex;non-Hopf}  below.

(b) A similar statement as (i)$\Rightarrow$(iii) can also be found in \cite{Mi}.
\end{rem}

We prove the theorem in several steps through the following lemmas.

\begin{lem}\label{l:locfinite} Let $(X,w,\mu)$ be a locally finite weighted graph and $d$ be a path pseudo metric.  Then, the following hold: 
\begin{itemize}
\item [(a)] $(X,d)$ is a discrete metric space. In particular, $(X,d)$ is locally compact.
\item [(b)]  A set is compact in $(X,d)$ if and only if it is finite.
\end{itemize}
\end{lem}
\begin{proof}
Local finiteness and the assumption $\si(x,y)>0$, $x\sim y$, imply that for all $x\in X$ there is an $r>0$ such that $d(x,y)>r$ for all $y\in X$ with $y\sim x$.
First, by the definition of $d$, we have that for all $x,z\in X$ there is $y\sim x$ with $d(x,y)\leq d(x,z)$. Thus, $d(x,y)=0$ implies $x=y$, therefore, $d$ is a metric. Second, it yields that  $B_{r}(x)=\{x\}$ and $\{x\}$ is an open set which shows (a). From this we conclude that  for any infinite set $U$ the cover $\{\{x\}\mid x\in U\}$ has no finite subcover. The other direction of (b) is clear.
\end{proof}


The authors are grateful to Florentin M\"unch for a crucial idea in the proof of the following lemma.

\begin{lem}\label{l:locfinite2} Let $(X,w,\mu)$ be a locally finite weighted graph and $d$ be a path metric. Assume that $B_{r}$ is infinite for some $r\geq0$. Then, there exists an  infinite geodesic of bounded length.
\end{lem}
\begin{proof}
 Let $o \in X$ be the center of the infinite ball $B_{r}$ of radius $r$ and let $d_{N}$ be the natural graph distance.
Let $P_{n}$, $n\geq0$,  be the set of finite paths $(x_{0},\ldots,x_{k})$ such that  $x_{0} = o$, $x_i \not = x_j$ for $i \not = j$, $d_{N}(x_{k},o)=n$ and $d_{N}(x_j,o)\leq n$ for $j=0,\ldots,k$.

\textit{Claim:}  $\Gm_{n}=\{\gm\in P_{n}\mid \mbox{$\gm$ geodesic with respect to $d$, } l(\gm)\leq r\}\neq\emptyset$ for all  $n\geq0$.\\
\textit{Proof of the claim:} The set $P_{n}$ is finite by local finiteness of the graph and, thus, contains a minimal element $\gm=(x_{0},\dots,x_{K})$ with respect to the length $l$, i.e., for all $\gm'\in P_{n}$ we have
$l(\gm')\ge l(\gm)$. Then, $\gm$ is a geodesic: for every path $(x_{0}',\ldots,x_{M}')$ with $x_{0}'=o$ and $x'_{M}=x_{K}$, we let $m\in\{n,\ldots,M\}$ be such that $(x_{0}',\ldots,x_{m}')\in P_{n}$.  By the minimality  of $\gm$ we infer
\begin{align*}
    l((x_{0}',\ldots,x_{M}'))\geq l((x_{0}',\ldots,x_{m}'))\ge l(\gm).
\end{align*}
It follows that $\gm$ is a geodesic.
Clearly, $l(\gm)\leq r$, as otherwise $B_{r}\subseteq \{y\in X\mid d_{N}(y,o)\leq n-1\}$ which would imply the finiteness of $B_{r}$ by the local finiteness of the path space.
Thus, $\gm\in\Gm_{n}$ which proves the claim.


We inductively construct an infinite geodesic $(x_{k})$ with bounded length: We set
$x_{0} = o$. Since $\Gm_{n}\neq\emptyset$, there is a geodesic in $\Gm_{n}$  for every $n \geq0$ such that $x_{0}$ is a subgeodesic. Suppose we have constructed a geodesic $(x_{1},\ldots,x_{m})$ such that for
all $n \geq m$ there is a geodesic in $\Gm_{n}$ that has $(x_{1},\ldots,x_{m})$ as a subgeodesic. By local
finiteness $x_{m}$ has finitely many neighbors. Thus, there must be a neighbor $x_{m+1}$ of $x_{m}$ such that for infinitely many $n$ the path
$(x_{0},\ldots,x_{m},x_{m+1})$ is a subpath of a geodesic in $\Gm_{n}$.
However, a subpath of geodesic is a geodesic.  Thus, there  is an infinite geodesic $ \gm=(x_{k})_{k\ge0}$ with $l(\gm)=\lim_{n\to\infty}l((x_{0},\ldots,x_{n}))\leq r$ as $(x_{0},\ldots,x_{n})\in \Gm_{n}$ for all $n\geq0$.
\end{proof}

\begin{proof}[Proof of Theorem~\ref{t:locfinite}]
The fact that $(X,d)$ is a discrete metric space follows from Lemma~\ref{l:locfinite}. We now turn to the proof of the equivalences. We start with (i)$\Rightarrow$(ii). If there is a bounded geodesic, then it is a Cauchy sequence. Since a geodesic is a path it is not eventually constant, thus it does not converge by discreteness. Hence, $(X,d)$ is not metrically complete. To prove (ii)$\Rightarrow$(iii) suppose that there is a distance ball that is infinite.
By Lemma~\ref{l:locfinite2} there is a bounded infinite geodesic and $(X,d)$ is not geodesically complete.
From Lemma~\ref{l:locfinite} (b) we deduce (iii)$\Leftrightarrow$(iv). Finally, we consider the direction (iv)$\Rightarrow$(i). If every bounded and closed set is compact, then every closed distance ball is compact. Then, by  Lemma~\ref{l:locfinite}~(b) every distance ball is finite and it follows that $(X,d)$ is metrically complete.
\end{proof}

We finish this appendix by giving several (counter-)examples to show that some of the statements above fail to be true in the case of non-locally finite graphs. We present the examples with respect to the path metric with
$\si = \si_{0}$ (see Example~\ref{adapted}).
Another example of this type can be found in \cite[Example~14.1]{FLW}.

\begin{eg} \label{Ex;non-Hopf}
Let $\mu\equiv1$, $\si = \si_{0}$, and $d=d_\si$.

(1) \textit{A metrically and geodesically complete graph with non compact distance balls.} \\
This example can be thought of as a star graph, where the rays are two subsequent edges.
Let $X=\N_{0}$ and let $w$ be symmetric with $w(0,2n)=1/2^{n}$ and $w(2n-1,2n)=1-1/2^{n}$ for $n\in\N$ and $w\equiv0$ otherwise. We have $d(0,2n)=1$ for $n\in\N$. Then, $(X,d)$ is metrically (and geodesically) complete but $B_{1}(0)$ is not compact. 

(2) \textit{A non locally compact graph.} \\
This example can be thought of as a star graph where the rays are copies of $\N$ whose lengths become shorter.
Let $X=\N_{0}^{2}$ and let $w$ be symmetric with $w((0,0),(m,0))=1/2^m$ for $m \in \N$ and $w((m,n-1),(m,n))=2^{2(m+n)}/5$ for $m, n \in \N$ and $w\equiv0$ otherwise.
Then, $\mathrm{Deg}((0,0))=1$, $\Deg((m,0)) = 1/2^m + 2^{2(m+1)}/5$ and $\mathrm{Deg}((m,n))=2^{2(m+n)}$ for $m, n \in \N$. Hence, we have $d((m,n-1),(m,n))=2^{-(m+n)}$ for $m,n \in \N$
and $1/2^{m+1}\leq d((0,0),(m,n))\leq 3/ 2^{m+1}$.
For a ball $B_{r}((m,n))$, $m,n\ge0$, $r>0$, denote by $U_{r}((m,n))$ its interior. Now for $\eps>0$, $\{U_{\eps/2}((0,0))\}\cup \{ U_{1/ 2^{(m+n+1)}}((m,n))\mid m \geq 1, n\geq0\}$ is an open cover of $B_{\eps}((0,0))$  with no finite subcover.


(3) \textit{A non Hausdorff space.}\\
This example can be thought as two vertices which are connected by infinitely many paths that become shorter.
Let $X=\N_{0}\cup\{\infty\}$ and let $w$ be symmetric with $w(0,2n)=w(\infty,2n)=1/2^{n}$ and $w(2n-1,2n)=2^{2n}$ and $w(n,m)=0$ all other $m,n\in\N_{0}$. Then, $\sigma(0,2n)=\sigma(\infty,2n)\leq1/2^{n}$. Hence, $d(0,\infty)=0$.

(4) \textit{An infinite ball and non discreteness.}\\
This example is a modification of (1).
Let $X=\N_{0}$ and let $w$ be symmetric with $w(0,2n)=1/2^{n}$ and $w(2n-1,2n)=2^{n}$ and $w(n,m)=0$ all other $m,n\in\N_{0}$. Then, every $d$-ball about $0$   is compact but it contains infinitely many vertices. Moreover, the vertices $x_{n}=2n$ converge to $x=0$ with respect to $d$. (This is, in particular, a counterexample to Lemma~\ref{l:locfinite} for non locally finite graphs).

(5) \textit{A geodesically complete graph which is not metrically complete.}\\
This example is an extension of (3) and can be thought as a ``line graph'' where between each two points on the line there are infinitely many ``line segments'' that become shorter.
Let $X = \N_0^2$ and let $w$ be symmetric with $w((m,0),(m,2n)) = 1/ 2^n = w((m+1,0),(m,2n))$ and $w((m,2n),(m,2n-1)) = 2^{2(m+1)} - 3 / 2^n$ for $m \in \N_0, n \in \N$ and $w \equiv 0$ otherwise.  It follows that $\Deg((m,2n)) = 2^{2(m+1)}- 1/2^n$ implying that $d((m,0),(m+1,0)) = 1/2^m$.  Thus $(x_m)= ((m,0))$ is a Cauchy sequence which does not converge.  On the other hand, the space is geodesically complete as there are no infinite geodesics.

\end{eg}

\textbf{Acknowledgements.}  M.K. enjoyed various inspiring discussion with Daniel Lenz and gratefully acknowledges the
financial support from the German Research Foundation (DFG).  R.K.W. thanks J{\'o}zef Dodziuk for numerous insights and acknowledges the financial support of the FCT under project PTDC/MAT/101007/2008 and of the PSC-CUNY Awards, jointly funded by the Professional Staff Congress and the City University of New York.  The authors are grateful to Ognjen Milatovic for a careful reading of the manuscript and suggestions.

\end{document}